\documentclass[reqno]{amsart}
\usepackage{amsmath,amssymb,color}
\usepackage[mathscr]{euscript}

\newtheorem{theorem}{Theorem}[section]
\newtheorem{lemma}[theorem]{Lemma}
\newtheorem{proposition}[theorem]{Proposition}

\newtheorem{remark}[theorem]{Remark}

\newcommand{\Glimsup}{\mathop{\textrm{$\Gamma$-$\limsup$}}\displaylimits}
\newcommand{\Gliminf}{\mathop{\textrm{$\Gamma$-$\liminf$}}\displaylimits}

\numberwithin{equation}{section}

\newcommand{\id}{{1 \mskip -5mu {\rm I}}}

\newcommand{\mc}[1]{{\mathcal #1}}

\newcommand{\bb}[1]{{\mathbb #1}}

\renewcommand{\epsilon}{\varepsilon}
\renewcommand{\phi}{\varphi}

\renewcommand{\hat}{\widehat}

\newcommand{\upbar}[1]{\,\overline{\! #1}}

\newcommand{\f}{\varphi}
\newcommand{\R}{\bb R}

\definecolor{light}{gray}{.9}

\title[Stationary solutions of Burgers equation]{A variational
  approach to the stationary solutions of Burgers equation}

\author [L.\ Bertini] {Lorenzo Bertini}
\address{\noindent Lorenzo Bertini \hfill\break\indent 
Dipartimento di Matematica, Universit\`a di Roma `La Sapienza' 
\hfill\break\indent 
P.le Aldo Moro 2, 00185 Roma, Italy}
\email{bertini@mat.uniroma1.it}

\author [M. Ponsiglione] {Marcello Ponsiglione}
\address{\noindent Marcello Ponsiglione\hfill\break\indent 
Dipartimento di Matematica, 
Sapienza  Universit\`a di Roma 
\hfill\break\indent 
P.le Aldo Moro 2, 00185 Roma, Italy}
\email{ponsigli@mat.uniroma1.it}

\begin{document}

\keywords{Burgers equation, Quasi-potential, Variational convergence}

\subjclass[2000]{Primary 35A15, 34B40; Secondary 35B40, 82B24}

\begin{abstract}
  Consider the viscous Burgers equation on a bounded interval with
  inhomogeneous Dirichlet boundary conditions.  
  Following the variational fra\-mework introduced in \cite{bd}, we
  analyze a Lyapunov functional for such equation which gives the
  large deviations asymptotics of a stochastic interacting particles
  model associated to the Burgers equation. We discuss the asymptotic
  behavior of this energy functional, whose minimizer is given by the
  unique stationary solution, as the length of the interval diverges.
  We focus on boundary data corresponding to a standing wave solution
  to the Burgers equation in the whole line. In this case, the
  limiting functional has in fact a one-parameter family of minimizers
  and we analyze the so-called development by $\Gamma$-convergence;
  this amounts to compute the sharp asymptotic cost corresponding to a
  given shift of the stationary solution.
\end{abstract}

\maketitle
\thispagestyle{empty}

\section{Introduction}
\label{sec0}

Consider the viscous Burgers equation on the interval $(a,b)\subseteq
(-\infty,+\infty)$ with inhomogeneous Dirichlet boundary conditions at
the endpoints
\begin{equation}
  \label{eq:1}
  \begin{cases}
    u_t + f(u)_x =  u_{xx}\,, 
    \\
    u(t,a) = u_- \,,\quad u(t,b) = u_+ 
    \;,
  \end{cases}
\end{equation}
where $u=u(t,x)$ satisfies $0\le u \le 1$, the \emph{flux} $f$ is the
function defined by $f(u):= u(1-u)$ and the boundary data satisfy
$0<u_-<u_+<1$. 
If the interval $(a,b)$ is bounded, it is simple to show that there
exists a unique stationary solution of \eqref{eq:1} that can be
computed explicitly. On the other hand, if we consider the case
$(a,b)=(-\infty,+\infty)$ and $u_-+u_+=1$, there exists a
\emph{standing wave} solution of Burgers equation on the whole line.
Accordingly, in this case  problem \eqref{eq:1} admits a one
parameter family of stationary solutions $\{\upbar{u}^{(z)}\,,\:z\in
\bb R\}$ which is obtained by considering the translations of the
standing wave $\upbar{u}$ satisfying
$\lim_{x\to\pm\infty}\upbar{u}(x)=u_\pm$ and $\upbar{u}(0)=1/2$.  
If we consider the case $u_-+u_+=1$ in the bounded symmetric interval
$(a,b)=(-\ell,\ell)$ and denote by $\upbar{u}_\ell$ the unique
stationary solution to \eqref{eq:1}, as $\ell$ diverges the sequence
$\{\upbar{u}_\ell\}$ converges to the stationary solution
$\upbar{u}$. We refer to \cite{GK} for a dynamical analysis of
\eqref{eq:1}.

The main topic we here discuss is the following.  Consider the case
$u_-+u_+=1$ and fix $z\in\bb R$. If we take $\ell$ large we expect
that there exists some function $u_\ell^{(z)}$ close to
$\upbar{u}^{(z)}$ such that $u_\ell^{(z)}$ is ``almost'' a stationary
solution to \eqref{eq:1} in the interval $(-\ell,\ell)$.  We shall
quantify being an ``almost'' stationary solution in terms of a
suitable family of energy functionals measuring the probability of
observing such fluctuations and compute the sharp asymptotic of the
energy of $u^{(z)}_\ell$.

Since the Burgers equation is not a gradient flow, the choice of the
energy functional is not trivial.  Consider the problem \eqref{eq:1}
in the bounded interval $(a,b)$ and denote by $V_{a,b}$, which also
depends on $u_\pm$, an associated energy functional. Let us first make
a short list of the properties that $V_{a,b}$ should enjoy:
\begin{itemize}
\item[\emph{(i)}]
the unique minimizer of $V_{a,b}$ is the stationary solution to
\eqref{eq:1};
\item[\emph{(ii)}] $V_{a,b}$ is a Lyapunov functional for the flow
defined by \eqref{eq:1};
\item[\emph{(iii)}] as $(a,b)$ diverges the functional $V_{a,b}$ converges to
  the functional $V_{-\infty,+\infty}$ associated to \eqref{eq:1} in
  the whole line.
\end{itemize}
Of course, this list still gives a lot of freedom.  However, as we
next discuss, there is a natural way to meet the requirements \emph{(i)} and
\emph{(ii)} above with some energy functional $V_{a,b}$ that has a clear
interpretation in terms of large deviations, while property \emph{(iii)} will
be proven in this paper.

First we associate to \eqref{eq:1} an action functional $I_{a,b}$,
defined on functions depending on space and time.  To this end, add an
external ``controlling'' field $E=E(t,x)$ to obtain the perturbed
equation
\begin{equation}
  \label{pbe}
  \begin{cases}
    v_t + f(v)_x +2\big( \sigma(v) E\big)_x =  v_{xx}, 
    \\
    v(t,a) = u_- \,,\quad v(t,b) = u_+ 
    \;,
  \end{cases}
\end{equation}
where $\sigma$ is a given positive function which, regarding $v$ as a
density, can be interpreted as the mobility of the system.  Denote by
$v^E$ the solution of this equation.  The action of a path
$v:(-\infty,0]\times [a,b]\to [0,1]$ is given by
\begin{equation}
\label{f14}
I_{a,b}(v) \;=\; \inf
\int_{-\infty}^0 \int_a^b \sigma(v) \, E^2 \, dx\,dt\;,
\end{equation}
where the infimum is carried over all $E$ such that $v^E=v$.  
Note that if $v$ is a solution to \eqref{eq:1} then $I_{a,b}(v)=0$. 
Consider now the so-called \emph{quasi-potential} \cite{FW} associated
to the action functional $I_{a,b}$, that is let $V_{a,b}$ be the
functional on the set of functions $u: [a,b] \to [0,1]$ defined by
\begin{equation}
\label{qpgen}
V_{a,b}(u) \;=\; \inf\, \big\{ I_{a,b}(v)\,:\: v(0)=u\,,\:
v(t)\to\upbar{u}_{a,b} \textrm{ as } t\to -\infty \big\}\;,
\end{equation}
where $\upbar{u}_{a,b}$ is the stationary solution to \eqref{eq:1}.
Namely, $V_{a,b}(u)$ is the minimal action to reach $u$ starting from
$\upbar{u}_{a,b}$. Of course, $V_{a,b}\ge 0$ and
$V_{a,b}(\upbar{u}_{a,b})=0$; it is also simple to check that
$V_{a,b}$ is a Lyapunov functional for \eqref{eq:1}.

The functional $V_{a,b}$ obtained by the previous general recipe
depends on the choice of the mobility $\sigma$. If the boundary data
are equal, $u_-=u_+=u_\circ$, then the stationary solution is
constant, $\upbar{u}=u_\circ$. In this case, it can be shown that the
quasi-potential is given by
\begin{equation}
  \label{vom}
  V_{a,b}(u) \;=\; \int_a^b s_{u_\circ}(u) \, dx \:, 
\end{equation}
where $s_{u_\circ}:[0,1]\to [0,+\infty)$ is the convex function such
that $s_{u_\circ}''(u)=1/\sigma(u)$ and $s_{u_\circ}(u_\circ) =
s_{u_\circ}'(u_\circ)=0$. Referring to \cite{BCM} for the proof of
\eqref{vom} in the case of periodic boundary conditions, we simply
observe that the functional in \eqref{vom} trivially satisfies the
requirements \emph{(i)}, \emph{(ii)}, and \emph{(iii)} above.

In the case of inhomogeneous boundary data $u_-\neq u_+$, the
quasi-potential $V_{a,b}$ is in general a nonlocal functional and, as
its definition involves the solution of a difficult dynamical problem,
its direct analysis does not appear feasible.  For the specific case
of the Burgers equation here considered and the choice
$\sigma(v)=v(1-v)$, in \cite{bd} it is shown that the quasi-potential
$V_{a,b}$ can be expressed in terms of a much simpler variational
problem which requires to optimize over functions of a single variable
rather than on all paths as in \eqref{qpgen}.

When $\sigma(v)=v(1-v)$ the action functional $I_{a,b}$ introduced in
\eqref{f14} is the dynamical large deviations rate functional of a
much studied stochastic model of interacting particles, the so-called
weakly asymmetric simple exclusion process \cite{KOV}.  Accordingly,
the quasi-potential $V_{a,b}$ describes the asymptotic behavior of the
corresponding invariant measure \cite{FW}.  More precisely, if we
denote by $\mu^N_{a,b}$ the invariant measure of the stochastic particles
model in the interval $(a,b)$ with lattice spacing $1/N$, then as
$N\to+\infty$ we have
\begin{equation}
  \label{ld}
  \mu^N_{a,b}( \mc B ) 
  \;\asymp\; \exp\big\{ - N \inf_{u\in\mc B} V_{a,b} (u)\big\}\;, 
\end{equation}
where $\mc B$ is a measurable subset of the configuration space.  In
particular, the probability on the left hand side converges to one as
$N\to \infty$ only if the global minimizer of $V_{a,b}$ lies in the
set $\mc B$. If otherwise $\upbar{u}_{a,b}\not\in \mc B$ the large
deviation formula \eqref{ld} expresses the fact that the probability
of $\mc B$ converges to zero exponentially fast in $N$ with rate given
by the infimum of $V_{a,b}$ on the set $\mc B$.  Within this context,
which takes into account the effect of fluctuations, we are thus
interested not only to the global minimizer of $V_{a,b}$, but also to
its minimizers in subsets of the function space.  A natural question is
the asymptotic behavior of the probability $\mu^N_{a,b}$ in the joint
limit in which $N\to\infty$ and the interval $(a,b)$ diverges.  A
simple approach to this issue, which corresponds to take first the
limit $N\to\infty$ and then letting the interval $(a,b)$ diverge, is
to analyze the variational convergence of $V_{a,b}$.

As the results in \cite{bd} are the starting point of the present
analysis, we briefly recall the main statement.  Given $p\in [0,1]$,
set $s(p)=p\log p +(1-p)\log(1-p)$. Note that $s''(p)=1/[p(1-p)]$ so
that $s$ can be regarded as the entropy function of the homogeneous
system; in the stochastic setting this function emerges naturally as
the Bernoulli entropy.  To the boundary data $0<u_-<u_+<1$ there
correspond the chemical potentials $\phi_\pm=s'(u_\pm)= \log [
{u_\pm}/({1-u_\pm})]\in \bb R$.  Given the bounded interval $[a,b]$
define the functional $\mc G_{a,b}$ of the two variables $u=u(x)$ and
$\phi=\phi(x)$, $x\in [a,b]$, as
\begin{equation}
  \label{Gint}
  \mc G_{a,b}(u,\phi) \;=\;
  \int_a^b \big[ 
  s(u) +  s(\phi') + (1-u) \phi -\log\big(1+e^\phi\big)
  \big]\, dx\;,
\end{equation}
where $\phi$ satisfies $0\le \phi'\le 1$ as well as $\phi(a)=\phi_-$
and $\phi(b)=\phi_+$.  Optimize now in $\phi$ to get a functional $\mc
F_{a,b}$ of $u$
\begin{equation}
  \label{Fint}
  \mc F_{a,b} (u) \; =\; 
  \inf_\phi \, \mc G_{a,b} (u,\phi)\;.
\end{equation}
In \cite{bd} it is proven that, apart an additive constant, the
quasi-potential is equal to $\mc F_{a,b}$, namely
\begin{equation}
  \label{V=F}
  V_{a,b} (u) \;=\; \mc F_{a,b}(u) \; -\; \inf \,\mc F_{a,b} \:. 
\end{equation}
Observe that the boundary data $u_\pm$ are passed to $\mc F_{a,b}$
thought the auxiliary function $\phi$. 
In particular, while the functional $\mc F_{a,b}$ is bounded on the
whole $L^\infty((a,b);[0,1])$, its minimizer is smooth and satisfies
the boundary conditions in \eqref{eq:1}.

The aim of this paper is to analyze the asymptotic behavior, in terms
of $\Gamma$-convergence \cite{Braides,Da}, of the functionals $\mc
F_{a,b}$ when the boundary data $u_\pm$ are fixed and the interval
$(a,b)$ diverges.  Although the functionals are quite different, some
of the arguments in the proofs of our results are similar to the ones
used in the analysis of the analogous problem for the van der Waals
free energy functional in a bounded interval \cite{bbb}.

In the case $u_-+u_+ >1$, the stationary solution to \eqref{eq:1} will
essentially make the transition from $u_-$ to $u_+$ close the left
endpoint $a$; note indeed that in this case the Burgers equation on
the whole line admits a travelling wave propagating towards the left.
We thus set $(a,b)=(0,\ell)$ and analyze the sequence of functionals
$\{\mc F_{0,\ell}\}$; we prove its $\Gamma$-convergence to a limiting
functional $\mc F_{0,+\infty}$ which is basically defined as in the case
of a bounded interval. 
In this situation, $\{\mc F_{0,\ell}\}$ has good coerciveness
properties to ensure the compactness of sequences with equibounded
energy.  Since the unique minimizer of $\mc F_{0,+\infty}$ is
given by the stationary solution $\upbar{u}_{0,+\infty}$ to
\eqref{eq:1} in the unbounded interval $(0,+\infty)$, the minimizer of
$\mc F_{0,\ell}$ converges  to $\upbar{u}_{0,+\infty}$.  Of course,
analogous results hold when $u_-+u_+ <1$.

In contrast, the case $u_-+u_+ =1$ is much richer. We consider the
symmetric interval $(a,b)=(-\ell,\ell)$ and analyze the asymptotic
behavior of the sequence of functionals $\{\mc F_{-\ell,\ell}\}$.  The
first result, that is the $\Gamma$-convergence to a limiting
functional $\mc F_{-\infty,+\infty}$, is analogous to the previous
case.  However, when $u_-+u_+ =1$, the functional $\mc F_{-\infty,+\infty}$ has a one
parameter family of minimizers, given by the stationary solutions to 
\eqref{eq:1} in the interval $(-\infty,+\infty)$.
For this reason, the sequence $\{\mc F_{-\ell,\ell}\}$ is not
equi-coercive: there are sequence $\{u_\ell\}$ such that $\mc
F_{-\ell,\ell} (u_\ell) \to \inf \mc F_{-\infty,+\infty}$ and
$z_\ell\to\infty$, where $z_\ell$ is the point such that
$u_\ell(z_\ell)=1/2$. We show that equi-coercivity of $\{\mc
F_{-\ell,\ell}\}$ is recovered if we identify functions that differ by
a translation.  In particular, since modulo translations $\mc
F_{-\infty,+\infty}$ has a unique minimizer, the shape of almost
minimizers for $\{\mc F_{-\ell,\ell}\}$ is rigid.

Let $\upbar{u}_\ell$ be the true minimizer of $\mc F_{-\ell,\ell}$.
As discussed before, the sequence $\{\upbar{u}_\ell\}$ converges to
$\upbar{u}$, the stationary solution of \eqref{eq:1} in the
interval $(-\infty,+\infty)$ such that $\upbar{u}(0)=1/2$.
The previous statement cannot be deduced from the $\Gamma$-convergence
of $\{\mc F_{-\ell,\ell}\}$. On the other hand, as it is customary in
those problems having a limiting functional with plenty of minimizers,
a variational description of this phenomenon is possible considering
the so-called development by $\Gamma$-convergence \cite{ab}.
More precisely, we introduce a rescaled excess energy $\mc
F^{(1)}_{-\ell,\ell}$   by setting 
\begin{equation}
  \label{dGc}
  \mc F^{(1)}_{-\ell,\ell} (u) 
  \;=\;   C(\ell)  
  \Big[ \mc F_{-\ell,\ell} (u) -\inf  \mc F_{-\infty,+\infty} \Big],
\end{equation}
and we look for a sequence $C(\ell)\to+\infty$ for which 
$\{ \mc F^{(1)}_{-\ell,\ell}\}$ has a non trivial $\Gamma$-limit.
Let $\alpha\in(0,1)$ be such that $u_\pm=(1\pm\alpha)/2$. We show that
the right choice of the rescaling is $C(\ell)=e^{\alpha\ell}$; this is
consistent with the fact that the stationary solution
$\upbar{u}$ approaches the asymptotic values $u_\pm$
exponentially fast. Then, we compute the corresponding  $\Gamma$-limit 
$\mc F^{(1)}_{-\infty,+\infty}$. Of course, $\mc
F^{(1)}_{-\infty,+\infty}(u)<+\infty$ only if $u=\upbar{u}^{(z)}$
for some $z\in\bb R$ and its explicit expression is given by 
\begin{equation}
  \label{f1exp}
  \mc F^{(1)}_{-\infty,+\infty}(u^{(z)}) \;= \; 
  \frac{8\alpha}{1-\alpha^2} \cosh(\alpha z). 
\end{equation}
In particular, since $\upbar{u}$ is the unique minimizer of $\mc
F^{(1)}_{-\infty,+\infty}$, the variational picture in terms of
development by $\Gamma$-convergence is complete.  
In general, the excess energy $e^{-\alpha\ell} \mc
F^{(1)}_{-\infty,+\infty}(\upbar{u}^{(z)})$ represents the cost for
shifting by $z$ the stationary solution $\upbar{u}_\ell$. In terms of
the large deviation formula \eqref{ld}, it gives the asymptotic
probability of a fluctuation close to $\upbar{u}^{(z)}$ as
$N\to+\infty$ and then $\ell\to+\infty$
\begin{equation}
  \label{ld1}
  \mu^N_{-\ell,\ell}( \mc O^{(z)} ) 
  \;\asymp \;\exp\Big\{ - N e^{-\alpha\ell} 
 \frac{8\alpha}{1-\alpha^2} 
  \big[ \cosh(\alpha z) -1 
  \big]
  \Big\} \:,
\end{equation}
where $\mc O^{(z)}$ is small neighborhood of $\upbar{u}^{(z)}$.

We finally briefly discuss the sharp interface setting.  This amounts
to the change of variable $x\mapsto x/\ell$, so that one considers the
Burgers equation \eqref{eq:1} in the fixed interval $(-1,1)$ with
viscosity $\epsilon=1/\ell$.  The asymptotic behavior of the energy
functionals can be clearly described also in the limit $\epsilon\to 0$,
see \cite{bd}.  In this setting stationary 
solutions to \eqref{eq:1} converge to step functions. Note that fluctuations
which are of order one in the unscaled variables are not seen
in the sharp interface limit.  
In particular, the $\Gamma$-limit \eqref{f1exp} of the rescaled excess
energy translated into the sharp interface setting becomes degenerate,
being infinite away from the minimizer. 
Even choosing a different rescaling in \eqref{dGc}, i.e.\ replacing
$e^{\alpha/\epsilon}$ with $e^{\beta/\epsilon}$, $\beta\in
(0,\alpha)$, we would still get a degenerate $\Gamma$-limit.
More precisely, with such a choice the $\Gamma$-limit would be zero
if the interface is at distance less than $1-\beta/\alpha$ from the
origin and infinite otherwise.

\section{The variational formulation}

In this section we introduce precisely the variational formulation for
stationary solutions to Burgers equation on bounded intervals and show
uniqueness of minimizers.
Fix a bounded interval $(a,b)\subset (-\infty,+\infty)$. Recalling that  the
flux is given by $f(u)=u(1-u)$, the stationary solution
$\upbar{u}_{a,b}$ to the viscous Burgers equation \eqref{eq:1} solves
the boundary value problem
\begin{equation}
  \label{ub}
  \begin{cases}
    u'' - [u(1-u)]'  =0 & x\in \bb (a,b)\;,  \\
    u(a) =u_- \,,\quad u(b) = u_+ \;,
  \end{cases}
\end{equation}
where $0<u_-<u_+<1$.
This problem admits a monotone solution, that satisfies the identity
\begin{equation}
  \label{curr}
  u' = u(1-u) -  J_{a,b} \:,
\end{equation}
where the \emph{current} $J_{a,b}$ is the constant determined by the
boundary conditions, i.e.\ it satisfies
\begin{equation}
  \label{condcurr}
  \int_{u_-}^{u_+} \frac{1}{r(1-r)-J_{a,b}} \, dr = b-a\;.
\end{equation}
Observe that $J_{a,b}$ is uniquely defined by the previous condition. 
Moreover $J_{a,b} < \min\{u_-(1-u_-),u_+(1-u_+)\}\le 1/4$ and $J_{a,b} >0$
as soon as $b-a > \log [u_+/(1-u_+)] -\log [u_-/(1-u_-)]$.  
In particular, from \eqref{curr} and \eqref{condcurr} we deduce
the uniqueness of the solution for the boundary value problem \eqref{ub}. 
Equation \eqref{curr} can be explicitly integrated getting
\begin{equation}
  \label{ubarab}
  \upbar{u}_{a,b} (x) = 
  \tfrac 12 + A_{a,b} \tanh \big[ A_{a,b} (x-x_{a,b}) \big]\;,
\end{equation}
where $A_{a,b} = \big(\frac 14 -J_{a,b}\big)^{\frac 12}$ and
$x_{a,b}\in \bb R$ is determined by imposing $\upbar{u}_{a,b}
(a)=u_-$.  In the sequel we consider only the cases $u_-+u_+=1$,
$(a,b)=(-\ell,\ell)$ and $u_-+u_+ \gtrless 1$, $(a,b)=\pm (0,\ell)$.
The corresponding solutions to \eqref{ub} are denoted by
$\upbar{u}_\ell$ and $\upbar{u}_\ell^\pm$, respectively.

The stationary solution to \eqref{eq:1} in the case of an unbounded
interval can be described analogously.
Given $0<u_-<u_+<1$ such that $u_-+u_+ \gtrless 1$, we consider the
boundary value problem
\begin{equation}
  \label{ubpm}
  \begin{cases}
    u'' - [u(1-u)]'  =0 & x\in  \bb R_\pm\;,  \\
    u(0)= u_\mp \;,\\
    \lim_{x\to \pm\infty} u(x) =u_\pm\;,
  \end{cases}
\end{equation}
whose solution $\upbar{u}^\pm$ is given by
\begin{equation}
  \label{ubarpm}
  \upbar{u}^\pm(x) =
    \tfrac 12 + A^\pm \tanh \big[ A^\pm (x-x^\pm) \big] \: ,
\end{equation}
where $A^\pm = \big|u_\pm - \frac 12\big|$ and $x^\pm\in\bb R$ is determined
by imposing $\upbar{u}^\pm (0)=u_\mp$.

Finally, given $0<u_-<u_+<1$ such that $u_-+u_+ = 1$, we consider the
boundary value problem
\begin{equation}
  \label{ubin}
  \begin{cases}
    u'' - [u(1-u)]'  =0 & x\in \bb R\;,  \\
    \lim_{x\to \pm\infty} u(x) =u_\pm\;,
  \end{cases}
\end{equation}
which has a one-parameter family  of solutions given by 
$\{\tau_z \upbar{u},\,z\in\bb R\}$ where $\tau_z$ is the translation
by $z$, i.e.\ $\tau_z u$ is the function defined by 
$(\tau_z u) \, (x) := u(x-z)$, and 
\begin{equation}
  \label{sol1}
  \upbar{u} (x) =
  \tfrac 12 + 
  \big(u_+ -\tfrac 12\big) 
  \tanh \big[\big(u_+ - \tfrac12\big) x \big] \;.
\end{equation}

\bigskip
We now introduce precisely the variational formulation 
\eqref{Gint}-\eqref{V=F} proposed in \cite{bd}.
Let $s:[0,1] \to \bb R$ be the convex function
\begin{equation}
  \label{s=}
  s(u) :=    u\log u +(1-u)\log(1-u) 
\end{equation}
where, as usual, we understand that $0\log 0=0$. 
Let also $g : [0,1]\times \bb R \times [0,1]  \to \bb R$ be the 
continuous function 
\begin{equation}
    \label{g=}
  g(u,\phi,p) := s(u) + s(p) + (1-u) \phi -\log\big(1+e^\phi\big)
  \;.
\end{equation}
Given $u_\pm \in (0,1)$ let 
\begin{equation}
  \label{phipm}
  \phi_\pm:=s'(u_\pm)= \log \frac {u_\pm}{1-u_\pm}   
\end{equation}
and observe that if $u_-+u_+=1$ then $ g(u_-,\phi_-,0) =
g(u_+,\phi_+,0)$. We set 
\begin{equation*}
  \begin{split}
    & \mc C^\pm :=\big\{ \phi\in AC(\bb R_\pm)\,:\: 0\le \phi' \le
    1\,, \: \phi(0)=\phi_\mp,\,
    \lim_{x\to\pm \infty}\phi(x)=\phi_\pm\big\}\;, \\
    & \mc C :=\big\{ \phi\in AC(\bb R)\,:\: 0\le \phi' \le 1\,, \:
    \lim_{x\to-\infty}\phi(x)=\phi_-\,,\:
    \lim_{x\to+\infty}\phi(x)=\phi_+\big\}\;,
\end{split}
\end{equation*}
where $AC$ denotes the space of absolutely continuous functions. We
consider $\mc C^\pm $ and $\mc C$ endowed with the topology of uniform
convergence, so that they are Polish spaces, i.e., metrizable,
complete and separable.
Then, we consider the spaces $L^\infty(\bb R_\pm;[0,1])$ and 
$L^\infty(\bb R;[0,1])$ endowed with the weak* topology, and set 
\begin{equation*}
\mc X^\pm := L^\infty(\bb R_\pm;[0,1]) \times
\mc C^\pm \;, \qquad \mc X := L^\infty(\bb
R;[0,1]) \times \mc C \;,
\end{equation*}
that we consider endowed with the product topology. Observe that also
$\mc X^\pm$ and $\mc X$ are Polish spaces.

If $u_-+u_+\gtrless 1$, equivalently $\phi_-+\phi_+ \gtrless 0$,
for each $\ell>0$ we let  $\mc G^\pm_\ell : \mc X^\pm \to (-\infty,+\infty]$
be the functional defined by 
\begin{equation}
  \label{Glpm}
   \mc G^\pm_\ell (u,\phi) :=
  \begin{cases}
    \displaystyle \int_{\pm (0,\ell)} 
    \big[ g(u,\phi,\phi) - g(u_\pm,\phi_\pm,0) \big] \, dx
    & \textrm{if $(u,\phi)\in B^\pm_\ell$} \: ,
    \\
    +\infty &\textrm{otherwise}\: ,
\end{cases}
\end{equation}
where 
\begin{equation*}
B^\pm_\ell :=\big\{ (u,\phi) \in \mc X^\pm :\: 
\phi(x)=\phi_\pm, \, u(x)=u_\pm \textrm{ for } 
x\in \pm [\ell,\infty) \big\}  \: .
\end{equation*}

If $u_-+u_+=1$, equivalently $\phi_-+\phi_+ =0$, for each
$\ell>0$ we let $\mc G_\ell : \mc X \to (-\infty,+\infty]$ be the
functional defined by
\begin{equation}
  \label{Gl}
  \mc G_\ell (u,\phi) :=
  \begin{cases}
    \displaystyle \int_{-\ell}^\ell 
    \big[ g (u,\phi,\phi') - g(u_+,\phi_+,0) \big] \, dx 
    & \textrm{if $(u,\phi)\in B_\ell$} \: ,
\\
+\infty &\textrm{otherwise}\: ,
\end{cases}
\end{equation}
where 
\begin{equation*}
  \begin{split}
    B_\ell := & \big\{ (u,\phi) \in \mc X :\: \phi(x)=\phi_-,
    \, u(x)=u_-  \textrm{ for } x \le - \ell\,,\:
    \\
    &\phantom{ \big\{ (u,\phi) \in \mc X :\: }
    \phi(x)=\phi_+, \, u(x)=u_+ \textrm{ for } 
    x \ge \ell \big\}\;.
  \end{split}
\end{equation*}

Observe that, up to an additive constant, the definition of 
$\mc G^\pm_\ell$ and $\mc G_\ell$ agrees with \eqref{Gint}.
According with \eqref{Fint}, we also define the functionals 
$\mc F^\pm_\ell:L^\infty(\bb R_\pm;[0,1]) \to (-\infty,+\infty]$ 
and 
$\mc F_\ell:L^\infty(\bb R;[0,1]) \to (-\infty,+\infty]$ 
by 
\begin{equation}
  \label{fpm}
  \mc F^\pm_\ell (u) := \inf_{\phi\in\mc C^\pm} \mc G^\pm_\ell(u,\phi) 
  \,,\qquad
  \mc F_\ell (u) := \inf_{\phi\in\mc C} \mc G_\ell(u,\phi) \;.
\end{equation}
Since $g$ is continuous and $p\mapsto g(u,\phi, p)$ is convex, the
above infima are attained as soon as $u(x)=u_\pm$ for $x \in \pm
[\ell,\infty)$, respectively $u(x)=u_-$ for $x\le -\ell$ and
$u(x)=u_+$ for $x\ge \ell$. As shown in \cite{bd}, there are functions
$u$ for which the set of minimizers is not a singleton.  Recalling
\eqref{ubarab}, we set $\upbar{\phi}^\pm_\ell
:=s'(\upbar{u}^\pm_\ell)$ and $\upbar{\phi}_\ell
:=s'(\upbar{u}_\ell)$.

\begin{proposition}
  \label{equafunc}
  For each $\ell>0$ the functionals $\mc G^\pm_\ell$ and
  $\mc G_\ell$ have a unique minimizer respectively given by 
  $(\upbar{u}^\pm_\ell,\upbar{\phi}^\pm_\ell)$ and 
  $(\upbar{u}_\ell,\upbar{\phi}_\ell)$.
  In particular, the unique minimizer of $\mc F^\pm_\ell$ and
  $\mc F_\ell$ is given by $\upbar{u}^\pm_\ell$ and $\upbar{u}_\ell$,
  respectively. 
\end{proposition}

\begin{proof}
  We prove the statement only for $\mc G_\ell$.
  In view of the strict convexity of $[0,1]\ni u\mapsto g(u,\phi,p)$, we
  can easily minimize $g(\cdot,\phi,p)$ and 
  the corresponding optimal $u$ satisfies $s'(u) = \phi$.
  Whence
  \begin{equation*}  
    \min_{u,\phi} \mc G_\ell(u,\phi) = 
    \min_\phi \mc G_\ell \big( (s')^{-1} (\phi),\phi \big)\;. 
  \end{equation*}

  The functional on the right hand side is clearly coercive and
  lower semicontinuous on $\mc C$.
  By the direct method of the calculus of variations, it
  thus admits a minimizer $\phi^*$.
  A straightforward computation shows that the Euler-Lagrange
  equation for $\mc G_\ell \big( (s')^{-1} (\phi),\phi \big)$ implies
  that $(s')^{-1} (\phi^*)$ solves \eqref{ub} in the interval 
  $(-\ell,\ell)$. By the uniqueness to such problem we deduce
  $\phi^*= s'(\upbar{u}_\ell) =\upbar{\phi}_\ell$. 

  Finally, the last statement follows from the coercivity
  of $\mc G_\ell$.
\end{proof}

\section{Variational convergence}
\label{sec2}

In this section we discuss the variational formulation 
on unbounded intervals. 
We show that the functionals in \eqref{Glpm} and \eqref{Gl} are well
defined also for $\ell = \infty$ and coincide with the $\Gamma$-limit
of the sequences $\{\mc G^\pm_\ell\}$ and $\{\mc G_\ell\}$ as
$\ell\to\infty$. 
In particular, this yields the stability of the boundary value
problems \eqref{ub}.

We start by the following proposition which yields the basic estimates
needed in sequel.  
Given $a_\pm\in \bb R$ we let $\vartheta_{a_-,a_+}(x) := a_- \id_{\bb
R_-}(x) + a_+ \id_{\bb R_+}(x)$.

\begin{proposition}
  \label{stimabasso}
  Let $u_-+u_+\gtrless 1$. There exists a constant $C> 0$ such that
  for any $(u,\phi)\in \mc X^\pm$
  \begin{equation*}
    \big\| g (u,\phi,\phi') - g(u_\pm,\phi_\pm,0) 
    \big\|_{L^1({\bb R_\pm} )}  
    \le C \, \big( \| u - u_\pm\|_{L^2({\bb R_\pm} )} 
    + \| \phi - \phi_\pm\|_{L^1({\bb R_\pm} )} + 1 \big)\;.
  \end{equation*}
  Let otherwise $u_- + u_+ = 1$. There exists a constant $C> 0$ such
  that for any $(u,\phi)\in \mc X$
  \begin{equation*}
    \big\| g (u,\phi,\phi') - g(u_+,\phi_+,0)
    \big\|_{L^1({\bb R} )} 
    \le C\, \big( \| u - \vartheta_{u_-,u_+}\|_{L^2({\bb R} )} + \| \phi -
    \vartheta_{\phi_-,\phi_+}\|_{L^1({\bb R} )} + 1 \big)\;.
  \end{equation*}
\end{proposition}

We premise an elementary lemma.

\begin{lemma}
  \label{t:el}
  Let $\phi_+, \, \f_- \in \bb R$ be such that $\phi_-<\phi_+$. For each
  $\gamma>0$ there exists a constant
  $C_\gamma(\phi_-,\phi_+)\in (0,+\infty)$, 
  satisfying 
  $C_\gamma(\phi_-,\phi_+)\to 0$ as $\phi_+-\phi_-\to 0$,
  such that   for any $\phi\in \mc C^+$
  \begin{equation*}
    \int_0^{+\infty} \big| s(\phi')\big| \, dx \le 
    \gamma \, \| \phi -\phi_+\|_{L^1(\bb R_+)} + C_\gamma (\phi_-,\phi_+)\;.
  \end{equation*}
\end{lemma}

\begin{proof}
  Given $\delta \in (0,1)$, let $A_\delta :=\big\{x\in \bb R_+\,:\:
  \phi'(x)\in [0,\delta]\big\}$ and set 
  $A_\delta^\complement:= \bb R_+\setminus A_\delta$. We write 
  \begin{equation}
    \label{decom}
     \int_0^{+\infty} \big| s(\phi')\big| \, dx
     =  \int_{A_\delta} \big| s(\phi')\big| \, dx
     +  \int_{A_\delta^\complement} \big| s(\phi')\big|\, dx
  \end{equation}
  and estimate separately the two terms on the right hand side. 
  To bound the first one, we first observe that for $p\in [0,1]$ we
  have  $|s(p)| \le p ( | \log p | + 1)$ and then 
  use H\"older inequality as follows 
  \begin{equation*}
    \begin{split}
      &\int_{A_\delta} \phi' | \log \phi'| \, dx  
       \;=\; \int_{A_\delta} \big[ \phi'(x) \, (1+ x) \big]^{2/3}
      \phi'(x)^{1/3} {(1 + x)}^{- 2/3} |\log \phi' (x) | \, dx \\
      & \qquad \le \Big[ \int_{A_\delta} \phi'(x) \, (1+x) \, dx
      \Big]^{2/3} \Big[ \int_{A_\delta} \phi'(x) \, (1+x)^{-2} \big|\log
      \phi' (x)\big|^3 \, dx \Big]^{1/3}
      \\
      & \qquad \le \eta_\delta \Big[ \int_{A_\delta} \phi'(x)\, (1+x)
      \, dx \Big]^{2/3} \le \eta_\delta \Big[ \int_0^\infty \phi'(x)
      \,x \, dx \: + \: \phi_+-\phi_-\Big]^{2/3}
      \\
      &\qquad
       = \eta_\delta \Big[ 
       \|\phi-\phi_+\|_{L^1(\bb R_+)} \: + \: \phi_+-\phi_-\Big]^{2/3} \:,
    \end{split}
  \end{equation*}
  where $ \eta _\delta := \max_{p\in[0,\delta]} p |\log p|^3  
  \, \int_0^\infty (1+x)^{-2}\,dx$ and in the last step we used that 
  $\|\phi-\phi_+\|_{L^1(\bb R_+)} = \int_0^\infty \phi'(x) \,x \, dx$. 
  Since $\int_{A_\delta} \phi' \, dx \le 
    \int_{0}^{+\infty} \phi' \, dx =\phi_+-\phi_-$, we get 
  \begin{equation}
    \label{primopezzo}      
    \int_{A_\delta} \big|s(\phi')\big| \, dx
    \;\le\; \eta_\delta \Big[ 
       \|\phi-\phi_+\|_{L^1(\bb R_+)} \: + \: \phi_+-\phi_-\Big]^{2/3} 
       \;+\; \phi_+-\phi_-
  \end{equation}
 
  To bound the second term on the right hand side of \eqref{decom}, we
  observe that, in view of the convexity of $s$, for each $p\in
  [\delta,1]$ we have 
  \begin{equation*}
    |s(p)|\le |s(\delta)| + |s'(\delta)|(p-\delta)
    \le |s(\delta)| + |s'(\delta)| p \;.
  \end{equation*}
  Denoting by $\big|A_\delta^\complement\big|$ the Lebesgue measure of
  $A_\delta^\complement$ we then deduce
  \begin{equation}
    \label{secondopezzo}
    \int_{A_\delta^\complement} \big| s(\phi'(x))\big|\, dx 
    \le  \int_{A_\delta^\complement} 
    \big[ |s(\delta)| +|s'(\delta)|\, \phi'(x) \big] \, dx
    \le |s(\delta)|  \big|A_\delta^\complement \big|
    +|s'(\delta)| (\phi_+-\phi_-)\:.
  \end{equation}
  Moreover, we have $\big|A_\delta^\complement \big|^2 \le 2\delta^{-1}
  \|\phi-\phi_+\|_{L^1(\bb R_+)}$. Indeed, 
  \begin{equation*}
    \begin{split}
      \big\|\phi-\phi_+\big\|_{L^1(\bb R_+)} 
      &\;=\; \int_0^{+\infty} \phi'(x)\,x\,dx 
      \;\ge\; \int_{A_\delta^\complement} \phi'(x)\,x\,dx 
      \\
      & \;\ge\; \delta  \int_{A_\delta^\complement} x\,dx 
      \;\ge\; \delta  \int_0^{| A_\delta^\complement|} x\,dx 
      \;=\; \tfrac {\delta}2 \, \big| A_\delta^\complement\big|^2 \;.
    \end{split}
  \end{equation*}
  Therefore, recalling \eqref{decom} and noticing that
  $\lim_{\delta\downarrow 0} \eta_\delta =0$,
  $\lim_{\delta\downarrow 0} \delta^{-1/2}s(\delta) =0$, the lemma
  follows easily from \eqref{primopezzo} and \eqref{secondopezzo}.
\end{proof}

\begin{proof}[Proof of Proposition \ref{stimabasso}]
  We prove the statement only for $u_- + u_+ >1$. Indeed,
  the statement for $u_- + u_+ <1$ is completely analogous,
  and the case $u_- + u_+ =1$ follows, noticing that
  $g(u_+,\phi_+,0)=g(u_-,\phi_-,0)$, from the previous ones.

  Recalling \eqref{g=} and $\phi_+ = s'(u_+)$, we write
  \begin{equation*}
    \begin{split}
      & g(u,\phi,\phi') - g(u_+,\phi_+,0) \;=\;  s(u) - s(u_+) - s'(u_+)(u
      - u_+) 
      \\ 
      &\qquad\quad
      + (1-u)(\phi-\phi_+) 
      -\big[ \log \big(1 + e^{\phi}
      \big)-\log \big(1 + e^{\phi_+}\big)\big] +s(\phi')\;.
    \end{split}
  \end{equation*}
  As $\phi\in \mc C^+$ implies $\phi\le \phi_+ = s'(u_+)$,
  \begin{equation}\label{4add}
    \begin{split}
      & \big| g(u,\phi,\phi') - g(u_+,\phi_+,0) \big | 
       \; \le \;  s(u) - s(u_+) - s'(u_+)(u - u_+)  \\ 
      &  \qquad\quad + (1-u)(\phi_+ - \phi) 
      + \log \big(1 + e^{\phi_+}\big) - \log\big(1 + e^{\phi}\big)
      + \big| s(\phi') \big|   \;.
    \end{split}
  \end{equation}
  Since $s$ is convex and $C^2$ in $(0,1)$, there exists a constant
  $C>0$ depending only on $u_+$ such that
  \begin{equation*}
    s(u) - s(u_+) - s'(u_+) (u - u_+) \le C  (u - u_+) ^2 \;.
  \end{equation*} 
  On the other hand, we clearly have $(1-u) (\phi_+ - \phi)\le |\phi -
  \phi_+|$.  Moreover, since the real function $\phi\mapsto \log(1 +
  e^{\phi})$ has Lipschitz constant one, we have
  \begin{equation*}
     \log\big(1 + e^{\phi_+}\big) - \log(1 + e^{\phi}) \le  |\phi - \phi_+| \;.
   \end{equation*} 
   In view of \eqref{4add}, the proof is now completed by applying
   Lemma~\ref{t:el}.
 \end{proof}

Let
\begin{equation*}
  \begin{split}
    &\mc D^\pm:= \{(u,\phi)\in \mc X^\pm \, \text{ such that } \| u
    - u_\pm\|_{L^2({\bb R_\pm} )} + \| \phi -
    \phi_\pm\|_{L^1({\bb R_\pm} )} < +\infty\}\;,\\ 
    &\mc D:= \{(u,\phi)\in \mc X \, \text{ such that } \| u -
    \vartheta_{u_-,u_+}\|_{L^2({\bb R} )} + \| \phi -
    \vartheta_{\phi_-,\phi_+}\|_{L^1({\bb R} )} < +\infty\}\;.
\end{split}
\end{equation*}
In view of Proposition~\ref{stimabasso}, we can introduce the functional $\mc
G^\pm : \mc X^\pm : (-\infty,+\infty]$ as follows
\begin{equation}
  \label{Gpm}
  \mc G^\pm (u,\phi) :=
  \begin{cases} 
    \displaystyle \int_{\bb R_\pm} \big[ g (u,\phi,\phi') -
    g(u_\pm,\phi_\pm,0) \big] \, dx
    & \textrm{ if } (u,\phi) \in \mc D^\pm \: ,  \\
    +\infty & \textrm{ otherwise}\;.
  \end{cases}
\end{equation}
Analogously, let $\mc G: \mc X : (-\infty,+\infty]$ be defined by 
\begin{equation}
  \label{G}
  \mc G (u,\phi) :=
  \begin{cases} 
    \displaystyle \int_{\bb R} 
    \big[ g (u,\phi,\phi') - g(u_+,\phi_+,0) \big] \, dx 
    & \textrm{ if } (u,\phi) \in \mc D  \: , \\
    +\infty & \textrm{ otherwise}\;.
\end{cases}
\end{equation}

Our first main result concerns the variational convergences of the
sequences $\{\mc G_\ell^\pm\}$ and $\{\mc G_\ell\}$, as respectively
defined in \eqref{Glpm} and \eqref{Gl}, as $\ell$ diverges.  The
appropriate notion is the so-called \emph{$\Gamma$-convergence}, see
\cite{Braides,Da}, that we next recall.  Let $X$ be a metric space and
$F_n : X\to (-\infty,+\infty]$, $n\in \bb N$.  The sequence of
functional $\{F_n\}$ is said to $\Gamma$-converge to $F: X\to
(-\infty,+\infty]$ iff the two following inequalities hold for any
$x\in X$:
\begin{itemize}
\item[\emph{(i)}] \emph{$\Gliminf$ inequality.}  For any sequence
  $\{x_n\} \subset X$ converging to $x$ we have $\liminf_n F_n(x_n)
  \ge F(x)$.
\item[\emph{(ii)}]\emph{$\Glimsup$ inequality.}  There exists a
  sequence $\{x_n\} \subset X$ converging to $x$ such that $\limsup_n
  F_n(x_n) \le F(x)$.
\end{itemize}
We also recall that the sequence $\{F_n\}$ is \emph{equi-coercive} iff 
any sequence $\{x_n\}\subset X$ such that $\limsup_n F_n(x_n)<+ \infty$
is precompact. 
As well known \cite{Braides,Da}, the $\Gamma$-convergence of a sequence of
equi-coercive functionals $F_n$ implies the convergence, up to a
subsequence, of their minimizers to a minimizer of the $\Gamma$-limit.

Recall that the stationary solutions to the Burgers equation
$\upbar{u}^\pm$ and $\upbar{u}$ are given in \eqref{ubarpm} and
\eqref{sol1}.  We set $\upbar{\phi}^\pm :=s'(\upbar{u}^\pm)$ and
$\upbar{\phi} :=s'(\upbar{u})$.

\begin{theorem}
  \label{t:1}
  Let $0<u_-<u_+<1$ be such that $u_-+u_+\gtrless 1$.
  \begin{itemize}  
  \item[(i)] There exists a constant $C\in (1,+\infty)$ such that for
    any $(u,\phi)\in \mc D^\pm$
    \begin{eqnarray}\label{stime1}
       \mc G^\pm(u,\phi) &\le& C\, \big(\|u - u_{\pm}\|_{L^2(\bb R_\pm)} + \|\phi -
        \phi_{\pm}\|_{L^1(\bb R_\pm)} \big)+ C,
      \\
      \label{stime2}
       \mc G^\pm(u,\phi) &\ge& \frac{1}{C}\, 
      \big(\|u - u_{\pm}\|_{L^2(\bb R_\pm)} + \|\phi - \phi_{\pm}\|_{L^1(\bb R_\pm)}
      \big)- C.
    \end{eqnarray}
    Moreover, the functional $ \mc G^\pm$ is lower semicontinuous and
    coercive on $\mc X^\pm$.  Finally, the unique minimizer of $\mc
    G^\pm$ is $(\upbar{u}^\pm,\upbar{\phi}^\pm)$.
  \item[(ii)] 
    The sequence of functionals $\{\mc G_\ell^\pm\}_{\ell>0}$ is
    equi-coercive and $\Gamma$-converges 
    to $\mc G^\pm$ as $\ell\to \infty$. In particular,
    $(\upbar{u}^\pm_\ell,\upbar{\phi}^\pm_\ell)\to
    (\upbar{u}^\pm,\upbar{\phi}^\pm)$ as $\ell\to\infty$.  
  \end{itemize}
\end{theorem}

In contrast to the previous case, in view of the translational
invariance of the limiting functional $\mc G$, the sequence $\{\mc
G_\ell\}$ is not equi-coercive. This loss of compactness takes place
because the ``interface'' between $u_-$ and $u_+$ can escape to
infinity with a bounded energy cost. However, as we state below, the
compactness of sequences with equibounded energy can be recovered if
we identify functions modulo translations. Recall that we denote by
$\tau_z$ the translation by  $z\in \bb R$.

\begin{theorem}
  \label{t:2}
  Let $0<u_-<u_+<1$ be such that $u_-+u_+=1$.
  \begin{itemize}
  \item[(i)] There exists a constant $C \in (1,+\infty)$ such that for any
    $(u,\phi)\in \mc D$
    \begin{eqnarray}\label{stime3}
      \mc G(u,\phi) &\le& 
      C\big(\|u - \vartheta_{u_-,u_+}\|_{L^2(\bb R)} 
      + \|\phi -\vartheta_{\phi_-,\phi_+}\|_{L^1(\bb R)} \big)+ C,
      \\
      \label{stime4}
      \mc G(u,\phi) &\ge& \frac{1}{C}\big(\|u - \tau_z
        \vartheta_{u_-,u_+}\|_{L^2(\bb R)} + \|\phi - \tau_z
        \vartheta_{\phi_-,\phi_+}\|_{L^1(\bb R)} \big)- C, 
    \end{eqnarray}
    for some $z\in\bb R$ depending on $(u,\phi)$. Moreover, the
    functional $ \mc G$ is lower semicontinuous on $\mc X$.  
    Finally, the set of minimizers of $\mc G$ is the one-parameter
    family of solutions to \eqref{ubin}.
  \item[(ii)] Let $\{(u_\ell,\phi_\ell)\}\subset \mc X $ and assume
    $\limsup_{\ell\to\infty} \mc G_\ell (u_\ell,\phi_\ell) <
    +\infty$.  Then there exists a sequence $\{z_\ell\}\subset \bb R$
    such that $\{(\tau_{z_\ell}
    u_{\ell},\tau_{z_\ell}\phi_{\ell})\}$ is precompact in $\mc X$.
    Moreover, the sequence of functionals $\{\mc G_\ell\}_{\ell>0}$
    $\Gamma$-converges to $\mc G$ as $\ell\to \infty$.  Finally,
    $(\upbar{u}_{\ell},\upbar{\phi}_{\ell}) \to (\upbar u,\upbar\phi)$
    as $\ell\to\infty$. 
 \end{itemize}
\end{theorem}

\begin{remark}
  \label{t:1r}
  {\rm Recall \eqref{fpm}, set $\mc F^\pm (u)=\inf_\phi \mc
    G^\pm(u,\phi)$ and $\mc F(u)=\inf_\phi \mc G(u,\phi)$.
    Theorems~\ref{t:1} and \ref{t:2}  imply the
    $\Gamma$-convergence of the sequences $\{\mc F^\pm_\ell\}$ and
    $\{\mc F_\ell\}$ to $\mc F^\pm$ and $\mc F$, respectively.  }
\end{remark}

\begin{proof}[Proof of Theorem \ref{t:1}]
  We prove the statements only in the case $u_-+u_+>1$.

  \smallskip
  \noindent \emph{Proof of \emph{(i)}.}  
  The upper bound \eqref{stime1} is a direct consequence of
  Proposition~\ref{stimabasso}. In order to prove the lower bound
  \eqref{stime2}, we first show that there exists
  $C_1=C_1(\phi_-,\phi_+)$ such that
  \begin{equation}
    \label{sl1}
    \mc G^+(u,\phi)  \ge \frac{1}{C_1}\int_0^{+\infty}\big( \phi_+ -
    \phi \big) \, dx - C_1\;. 
  \end{equation}
  Observe that $s' : (0,1)\to \bb R$ is given by
  $s'(p)=\log[p/(1-p)]$.  Hence $(s')^{-1}: \bb R \to (0,1)$ is given
  by $(s')^{-1}(q)=e^q/(1+e^q)$.  Therefore
  \begin{equation*}
    f(q) :=
    g\big((s')^{-1}(q),q,0\big)= q-2\log(1+e^q)  
  \end{equation*}
  and in particular, $g(u_+,\phi_+,0)= f(\phi_+)$.  By the
  strict convexity of $g(\cdot,q,p)$ for a fixed $(q,p)\in
  [\phi_-,\phi_+]\times [0,1]$, the infimum of $\mc G^+ (\cdot
  ,\phi)$ for a fixed $\phi \in \mc C^+$ is achieved when $u$
  satisfies $s'(u)=\phi$.  Hence
  \begin{equation*}
    \mc G^+(u,\phi) \ge \mc G^+ \big( (s')^{-1}(\phi), \phi \big) = 
    \int^\infty_0 \big[ s(\phi') + f(\phi)
    -f(\phi_+)\big]\,dx\;. 
  \end{equation*}
  Since the real function $f$ is concave, for any
  $q\in[\phi_-,\phi_+]$
  \begin{equation}\label{defemme}
    f(q) -f(\phi_+) \ge 
    \frac{f(\phi_-) -f(\phi_+)}{\phi_+-\phi_-} (\phi_+ -q)
    =: m (\phi_+ -q)\;. 
  \end{equation}
  It is simple to check that $m>0$ because $\phi_-+\phi_+ >0$.
  We thus deduce
  \begin{equation}
    \label{infimum}
    \begin{split}
    & \mc G^+(u,\phi)   
      \; \ge\; \frac m2 \int_0^\infty \big[ \phi_+- \phi \big]\,dx 
    \\
    &\qquad +  
    \inf \Big\{
    \int_0^\infty \big[ s(\psi') + \tfrac m2
    \big( \phi_+- \psi \big) \big]\,dx
    \,,\:  \psi\in \mc C^+ \,:\: \psi -\phi_+  \in {L^1(\bb R_+)} \Big\}  \: .
    \end{split}
 \end{equation}
  In view of Lemma~\ref{t:el}, the infimum on the right hand side
  above is finite. This concludes the proof of the bound \eqref{sl1}.

  We next prove the $L^2$ bound on $u$.  Since the right hand
  side of \eqref{infimum} is bounded from below, there exists a
  constant $C_2=C_2(\phi_-,\phi_+)$ such that for any $\phi\in\mc C^+$ 
  $\mc G^+\big( (s')^{-1}(\phi),\phi\big)\ge - C_2$. 
  Therefore 
  \begin{equation*}
    \begin{split}
      \mc G^+(u,\phi) & \; \ge \; \mc G^+(u,\phi) 
      - \mc G^+\big((s')^{-1}(\phi),\phi\big) - C_2
      \\
      & \; = \; \int_0^\infty \big\{ s(u) - \phi \,u - \big[ s\big(
      (s')^{-1}(\phi) \big) - \phi \: (s')^{-1}(\phi) \big]
      \big\}\, dx - C_2 \: .
    \end{split}
  \end{equation*}
  Since $u \in [0,1]$, $(s')^{-1}$ is locally Lipschitz on $\bb R$,
  $(s')^{-1}(\phi_+) = u_+ \in (0,1)$ and $s$ is locally Lipschitz
  in $(0,1)$, the $L^1$ bound on $\phi -\phi_+$ implies there
  exists $C_3$ such that
  \begin{equation*}
    \mc G^+(u,\phi) \;\ge\;  \int_0^\infty \big[ s(u) - s(u_+) - s'(u_+)
    (u-u_+) \big]\, dx - C_3 \| \phi - \phi_+\|_{L^1(\bb R_+)} - C_2\;.
  \end{equation*}
  The proof is completed by the uniform convexity of $s$ on $[0,1]$.

  To prove the lower semicontinuity of $\mc G^+$, given a sequence
  $\{(u_n,\phi_n)\} \subset \mc X^+$ converging to $(u,\phi)$, we need to
  show that $\mc G^+(u,\phi) \le \liminf_n \mc G^+(u_n,\phi_n)$. We can
  clearly assume that $\liminf_n \mc G^+(u_n,\phi_n)< +\infty$, and
  therefore, by taking if necessary a subsequence, that
  $\{(u_n,\phi_n)\}$ has equibounded energy.  In particular, by
  \eqref{stime2} we deduce that $u-u_+$ belongs to $L^2(\bb R^+)$ and
  $\phi-\phi_+$ belongs to $L^1(\bb R^+)$.  Again, from \eqref{stime2} we
  easily deduce that
  \begin{equation}\label{coerfi}
    \lim_{L\to\infty} \liminf_n \phi_n(L) = \phi_+.
  \end{equation}
  Given an interval $I\subseteq \R$, we introduce the localized
  functional $\mc G^+_{I}$ defined by
  \begin{equation}\label{localized}
    \mc G^+_{I}(u,\phi) := \int_I \big[ g (u,\phi,\phi') 
    - g(u_+,\phi_+,0) \big] \, dx \: .
  \end{equation}
  By the convexity of $s$, for each $L>0$ the functional
  $\mc G^+_{(0,L)}$  
  is lower semicontinuous on $\mc X^+$. Since by
  Proposition~\ref{stimabasso} $\lim_{L\to\infty} \mc
  G^+_{(L,+\infty)}(u,\phi) = 0$, to complete the proof it is thus
  enough to show that
  \begin{equation}
    \label{finalrelax}
    \lim_{L\to\infty} \liminf_{n} \mc G^+_{(L,+\infty)}(u_n,\phi_n) \ge 0 \:. 
  \end{equation}
  To this purpose, let $m$ be as defined in \eqref{defemme}. 
  Arguing as in the proof of \eqref{infimum} 
  \begin{equation*}
    \begin{split}
      \mc G^+_{(L,+\infty)}(u_n,\phi_n) & 
      \ge 
      \inf\Big\{ \int_L^\infty
      \big[ s(\psi') + m \big( \phi_+ - \psi \big) \big]\,dx
      \,,\;\; \psi\,:\: \psi(L) =\phi_n(L) \Big\} 
      \\
      & =
      \inf\Big\{ \int_0^\infty
      \big[ s(\psi') + m \big( \phi_+ - \psi \big) \big]\,dx
      \,,\;\; \psi\,:\: \psi(0) =\phi_n(L) \Big\} \: ,
  \end{split}
  \end{equation*}
  where, of course, $\psi$ is increasing and satisfies
  $\lim_{x\to+\infty}\psi(x)=\phi_+$.  In view of \eqref{coerfi} and
  Lemma~\ref{t:el} the bound \eqref{finalrelax} follows.

  The coercivity of $\mc G^+$ follows trivially from the
  equi-coercivity and $\Gamma$-convergence of the sequence $\{\mc
  G_\ell^+\}$ proven in item (ii) below. 

  Since $\mc G^\pm$ is bounded from below, coercive, and
  lower-semicontinuous, by arguing as in the proof of Proposition
  \ref{equafunc} we conclude that the unique minimizer of $\mc G^\pm$
  is the solution to \eqref{ubpm}.

  \smallskip
  \noindent \emph{Proof of \emph{(ii)}.}  
  Let $\{(u_\ell,\phi_\ell)\} \subset \mc X^+$ be a sequence such that
  $\mc G_\ell(u_\ell,\phi_\ell) \le K$ for some $K\in \bb R$; we next
  show that $\{(u_\ell,\phi_\ell)\} $ is precompact.  To this purpose,
  notice that by the very definition \eqref{Glpm} of $ \mc G_\ell$ we
  have $\mc G_\ell (u_\ell,\phi_\ell) = \mc G (u_\ell,\phi_\ell)$.
  Therefore, by the lower bound \eqref{stime2} we have
  \begin{equation}\label{coer}
    \| u_\ell - u_+\|_{L^2(\bb R^+)} +  \| \phi_\ell -
    \phi_+\|_{L^1(\bb R^+)} \le C(K+C). 
  \end{equation}

  Fix $\epsilon>0$ and let $x_{\epsilon,\ell} := \inf\{x>0\,,\:
  \phi_\ell(x)= \phi_+ -\epsilon\}$.  From \eqref{coer} we
  deduce $\limsup_\ell x_{\epsilon,\ell} < +\infty$.  Since
  $\phi_\ell' \in [0,1]$, we thus deduce the precompactness of
  $\{\phi_\ell\}$ in $\mc C^+(\bb R_+)$. As $L^\infty(\bb R_+;[0,1])$
  is compact with respect to the weak* topology this concludes the
  proof of the equi-coercivity of $\mc G_\ell$.

  In order o prove the $\Gamma$-liminf inequality let
  $(u_\ell,\f_\ell) \to (u,\f)$, and assume without loss of generality
  that $(u_\ell,\f_\ell)$ has equibounded energy. By the very
  definition \eqref{Glpm} of $\mc G_\ell$ and the lower semicontinuity of
  $\mc G$
  \begin{equation*}
    \mc G(u,\f) \le \liminf_\ell \mc G(u_\ell,\f_\ell) 
    =  \liminf_\ell \mc G_\ell(u_\ell,\f_\ell).
  \end{equation*}

  To prove the $\Gamma$-limsup inequality fix $(u,\f) \in \mc X^+$
  with finite energy. We define $(u_\ell,\f_\ell)$ as follows. We set $u_\ell
  \equiv u$ in $[0,\ell]$, and $u_\ell \equiv u_+$ in $(\ell,+\infty)$.
  Moreover, we set $\f_\ell \equiv \f$ in $[0, \ell-1]$, $\f_\ell \equiv \f_+$
  in $[\ell,+\infty)$, and we extend it by affine interpolation on
  $[\ell-1,\ell]$. A direct computation shows that $(u_\ell,\f_\ell)$ is a
  recovery sequence for $(u,\f)$.
\end{proof}

\begin{proof}[Proof of Theorem \ref{t:2}] 
  The proof will be easily achieved by applying Theorem~\ref{t:1} and
  translations invariance arguments.

  \smallskip
  \noindent \emph{Proof of \emph{(i)}.}  
  The upper bound \eqref{stime3} on $\mc G$ is a direct consequence of
  Proposition~\ref{stimabasso}. In order to prove the lower bound
  \eqref{stime4}, let $z\in \bb R$ be such that $\f(z) = (\f_+ +
    \f_{-})/{2} = 0$. The lower bound then follows by applying
  \eqref{stime2} with $\bb R^\pm$ replaced by $\{x\le z\}$ and $\{x\ge
  z\}$.

  We next prove the lower semicontinuity of $\mc G$.  Let $(u_n,\f_n)
  \to (u,\phi)$; by taking, if necessary, a subsequence we assume that
  $\liminf_n \mc G (u_n,\f_n) = \lim_n \mc G(u_n,\phi_n)$. Let
  $\{z_n\}\subset \bb R$ be such that $\phi_n(z_n)=0$. By taking, if
  necessary, a further subsequence we assume that $z_n\to z$ 
  for some $z\in\bb R$. According with the notation introduced in \eqref{localized}, we write
  \begin{equation*}
    \mc G(u_n,\f_n) = \mc G_{(-\infty, z_n)}(u_n,\f_n) 
    +  \mc G_{( z_n, +\infty)}(u_n,\f_n) \;.
  \end{equation*}
  By translation invariance the statement now follows from the lower
  semicontinuity of $\mc G^\pm$, see item (i) in Theorem~\ref{t:1}.

  The last statement will follow once we prove that the pair $(\upbar u,
  \upbar \f)$ is the unique minimizer of $\mc G$ among all $(u,\f)\in
  \mc X$ satisfying $\f(0) = 0$. This readily follows from the
  uniqueness property stated in item (i) of Theorem~\ref{t:1}.

  \smallskip
  \noindent \emph{Proof of \emph{(ii)}.}  
  Given a sequence $\{(u_\ell,\phi_\ell)\}\subset \mc X$
  such that $\mc G_\ell(u_\ell,\phi_\ell)<+\infty$, let
  $\{z_\ell\}\subset \bb R$ be such that $\phi_\ell(z_\ell)=0$.
  Observe that $z_\ell\in (-\ell,\ell)$ and  write 
  \begin{equation*}
    \begin{split}
         \mc G_\ell(u_\ell,\f_\ell) & = \mc G_{(-\ell, z_\ell)}(u_\ell,\f_\ell) 
    +  \mc G_{(z_\ell, \ell)}(u_\ell,\f_\ell) \\
    & = \mc G^-_\ell (\tau_{-z_\ell} u_\ell, \tau_{-z_\ell} \phi_\ell)
    + \mc G^+_\ell (\tau_{-z_\ell} u_\ell, \tau_{-z_\ell} \phi_\ell)
    \;.
    \end{split}
  \end{equation*}
  In view of the equi-coercivity of $\mc G_\ell^\pm$ stated in item
  (ii) Theorem~\ref{t:1}, the compactness property of $\{\mc G_\ell\}$
  follows. 
  Thanks again to item (ii) in Theorem~\ref{t:1}, the above decomposition
  of the energy functionals easily yields the $\Gamma$-convergence
  result.
  
  Finally, the fact that the unique minimizer $(\upbar{u}_{\ell},
  \upbar{\phi}_{\ell})$ of $\mc G_\ell$ converges to $(\upbar{u},
  \upbar{\phi})$ in $\mc X$ follows by the $\Gamma$-convergence and 
  the fact that $\upbar{\phi}_{\ell}(0) = 0$ for every $\ell$.
\end{proof}

\section{Development by $\Gamma$-convergence}
\label{sec3}

In this section we analyze in more detail the functional $\mc G_\ell$
for the choice of the boundary data corresponding to a standing wave
for the Burgers equation in the whole line.
As we discussed previously, when $u_-+u_+=1$ there is a one parameter
family of solutions to \eqref{ubin} given by $\tau_z \upbar{u}$, $z\in
\bb R$.
Accordingly, the minimizers of the limiting functional $\mc G$ in
\eqref{G} are given by $(\tau_z \upbar{u},\tau_z \upbar{\phi})$.  On the other hand, recalling Theorem~\ref{t:2}, the minimizer
of the finite volume energy $\mc G_\ell$ is unique and converges to
$(\upbar{u},\upbar{\phi})$.  The purpose of this section is to provide
a variational framework to select this minimizer, based again on the
notion of $\Gamma$-convergence, and more precisely on the notion of
development by $\Gamma$-convergence introduced in \cite{ab}.

Let $\alpha\in (0,1)$ be such that $u_\pm = (1\pm\alpha)/2$.
Recalling \eqref{Gl} and \eqref{G}, we are interested in the
asymptotic behavior of the functionals $\hat{\mc G}_\ell :\mc X \to
(-\infty,+\infty]$ defined by
\begin{equation}
  \label{hGl}
  \hat{\mc G}_\ell (u,\phi) 
  := e^{\alpha\ell} 
  \big[  \mathcal{G}_\ell (u,\phi)  - \min \mathcal{G} \big] \;,
\end{equation} 
where the exponential rescaling has been introduced to get a non
trivial limit. In particular, as stated below, the unique minimizer of
the limiting functional is $(\upbar{u},\upbar{\phi})$, where
$\upbar{u}$ is the stationary solution satisfying $\upbar{u}(0)=1/2$,
see \eqref{sol1}, and $\upbar{\phi}=s'(\upbar u)$.

\begin{theorem}
  \label{sviluppo}
  Let $\alpha\in (0,1)$ be such that $u_\pm=(1\pm\alpha)/2$.
 \begin{itemize}  
 \item[(i)] 
   Let $\{(u_\ell,\phi_\ell)\}\subset \mc X $ be sequence such that
   $\limsup_{\ell\to\infty} \hat{\mc G}_\ell(u_\ell,\phi_\ell)
   <+\infty$.  Then, up to a subsequence, $(u_\ell,\phi_\ell) \to
   (\tau_{z} \upbar{u},\tau_{z} \upbar{\phi})$ for some $z\in
   \bb R$. In particular, the sequence of functionals 
   $\{\hat{\mc G}_\ell \}$ is equi-coercive.
 \item[(ii)] As $\ell\to\infty$ the sequence of functionals 
   $\{\hat{\mc G}_\ell\}$ $\Gamma$-converges to the functional
   $\hat{\mc G} : \mc X \to (-\infty,+\infty]$ defined by 
   \begin{equation*}
    \hat{\mc G} (u,\f):=
   \begin{cases} 
     \displaystyle{ \frac {8\alpha}{1-\alpha^2}  \cosh(\alpha z) } 
     & \text{ if } (u,\phi) =(\tau_z \upbar u, \tau_z \upbar\phi)
     \text{ for some } z\in \bb R,\\ 
     +\infty & \text{ otherwise.}
   \end{cases}
 \end{equation*}
\end{itemize}
\end{theorem}

\begin{remark}
  \label{t:2r}
  {\rm Let $\hat{\mc F}_\ell (u)=\inf_\phi \hat{\mc G}_\ell(u,\phi)$.
    Theorem~\ref{sviluppo} then imply that $\{\hat{\mc F}_\ell\}$ 
    is equi-coercive and $\Gamma$-converges to 
    $\min_\phi \hat{\mc G}(u,\phi)$. Note that the $\Gamma$-limit is finite
    only if $u=\tau_z\upbar u$ for some $z\in\bb R$ and in such a case
    is given by $\hat{\mc G}(\tau_z \upbar u,\tau_z\upbar\phi)$.
  }
\end{remark}

The proof of Theorem~\ref{sviluppo} is based on few computations that
we collect in the following two lemmata.  Recall that $\upbar{u}_\ell$ is
the solution to \eqref{ub} in the symmetric interval $(-\ell,\ell)$
and $\upbar\phi_\ell=s'(\upbar u_\ell)$. In particular, $\upbar
u_\ell$ solves $\upbar u_\ell' = \upbar u_\ell (1-\upbar u_\ell)
-J_\ell$ where $J_\ell\in \bb R$ is the constant satisfying
\begin{equation}
  \label{Jel2}
  \int_{u_-}^{u_+} 
  \frac 1{r(1-r)-J_\ell} \, dr = 2\ell\;.
\end{equation}
Analogously, ${\upbar u}' = \upbar u(1-\upbar u) -J$ where
$J=(1-\alpha^2)/4$. It is simple to check that $J_\ell\uparrow J$ as
$\ell\to\infty$; in the sequel we need however the sharp asymptotic of
$J_\ell$.   

\begin{lemma}
  \label{t:jell}
  Let $\alpha\in (0,1)$ be such that $u_\pm=(1\pm\alpha)/2$. Then
  \begin{equation*}
    \label{Jelto2}
    \lim_{\ell\to \infty} e^{\alpha\ell} \big(J- J_\ell\big) = \alpha^2 \;.
  \end{equation*}
\end{lemma}

\begin{proof}
  The integral on the left hand side of \eqref{Jel2} can be calculated
  explicitly and the proof of the lemma can be achieved by few
  tedious computations. We give however a lighter argument below.

  Let $E_\ell : = J - J_\ell$, and notice that $E_\ell \downarrow 0$ as $\ell \to \infty$.  By
  symmetry, \eqref{Jel2} is equivalent to
  \begin{equation*}
    \int_{1/2}^{(1+\alpha)/2} 
    \frac 1{r(1-r)-J +E_\ell} \, dr = \ell \;.
  \end{equation*}
  Performing the change of variables $r=(1+\alpha-s)/2$, this is
  further equivalent to 
  \begin{equation}
    \label{Jel2bis}
    2 \int_{0}^{\alpha} 
    \frac 1{2\alpha s -s^2 + 4E_\ell} \, ds = \ell \;,
  \end{equation}
  where we used that $J=(1-\alpha^2)/4$. 
  We rewrite the left hand side above as follows 
  \begin{equation}
    \label{Jel2tris}
    \begin{split}
    2 \int_{0}^{\alpha} 
    \frac 1{2\alpha s -s^2 + 4E_\ell} \, ds
    &\;=\; 2 \int_{0}^{\alpha} 
    \frac 1{2\alpha s + 4E_\ell} \, ds + R(E_\ell)
    \\
    &\;=\;  \frac 1\alpha 
    \log \Big[ \frac{\alpha^2}{2E_\ell}
    \Big(1+ \frac {2E_\ell}{\alpha^2}\Big)\Big]
    + R(E_\ell) \: ,
    \end{split} 
 \end{equation}
 where
 \begin{equation*}
   R(E) :=   2 \int_{0}^{\alpha} 
   \Big[\frac 1{2\alpha s -s^2 +4E} -\frac 1{2\alpha s+ 4E}\Big]
   \, ds\;.
 \end{equation*}
 It is simple to check that
 \begin{equation*}
   R(0)=\lim_{E\downarrow 0} R(E)
   = 2 \int_{0}^{\alpha} 
   \frac {s^2}{2\alpha s (2\alpha s -s^2)}\, ds
   = \frac 1\alpha\log 2\;.
 \end{equation*}
 From \eqref{Jel2bis} and \eqref{Jel2tris} we then get 
 \begin{equation*}
   \ell=  
   \frac 1\alpha \log \frac{\alpha^2}{E_\ell}
   + \frac 1\alpha \log \Big(1+ \frac {2E_\ell}{\alpha^2}\Big)
    + R(E_\ell) - R(0) \: ,
 \end{equation*}
 which, recalling $E_\ell\downarrow 0$ as $\ell\to\infty$, yields 
 the statement.
\end{proof}

As follows from Theorem~\ref{t:2}, $\mc G_\ell(\upbar u_\ell,\upbar
\phi_\ell)\to \mc G (\upbar u,\upbar\phi)$ as $\ell\to \infty$.  We
next compute the sharp asymptotic of the energy of the minimizer
$(\upbar u_\ell,\upbar\phi_\ell)$.

\begin{lemma}
  \label{lemma}
  Let $\alpha\in (0,1)$ be such that $u_\pm=(1\pm\alpha)/2$. Then
  \begin{equation*}
    \lim_{\ell\to\infty} e^{\alpha\ell}
    \Big[ \mc G_\ell(\upbar u_\ell,\upbar \f_\ell) 
    -\mc G(\upbar u,\upbar \f) \Big]= 
    \frac {8\alpha}{1-\alpha^2}.
  \end{equation*}
\end{lemma}

\begin{proof} 
  Recall that $\upbar u_\ell$ and $\upbar u$ satisfy \eqref{curr}
  with constants $J_\ell$ and $J$, respectively. Recall also that
  $\upbar\phi_\ell=s'(\upbar u_\ell)$. 
  From the very definition \eqref{g=}  of $g$ we  deduce
  \begin{equation}
    \label{g=2}
    g \big(\upbar u_\ell,\upbar\phi_\ell,\upbar\phi_\ell' \big)
    = \log J_\ell
    + \frac {\upbar u_\ell'}{\upbar u_\ell (1-\upbar u_\ell)}
    \log \frac {\upbar u_\ell(1-\upbar u_\ell) - J_\ell}
    {J_\ell} \: ,
  \end{equation}
  both in the case $\ell\in (0,+\infty)$ and $\ell=+\infty$.  Note
  also that $g(u_+,\phi_+,0)=g(u_-,\phi_-,0)=\log J$.
  From \eqref{g=2} we get
  \begin{equation}
    \label{s3}
    \begin{split}
      &\mc G_\ell(\upbar u_\ell,\upbar \f_\ell) -\mc G(\upbar u,\upbar \f)
      \\
      &\qquad = 2\ell \, \log \frac {J_\ell}{J}
      + \int_{u_-}^{u_+} \frac{1}{r(1-r)} 
      \log\Big[  \frac { r(1-r) -J_\ell}{J_\ell} 
      \frac {J}{ r(1-r) -J}  \Big] \, dr 
      \\
      &\qquad = \int_{u_-}^{u_+} \frac{1}{r(1-r)} \Big[ \log
      \frac{r(1-r) - J_\ell}{r(1-r) - J} +\frac
      {J_\ell}{r(1-r)-J_\ell} \log\frac{J_\ell}{J} \Big] \,dr \: ,
    \end{split}
  \end{equation}
  where we used \eqref{Jel2} in the second step.

  Set $\epsilon_\ell := 4(J -J_\ell)$ and observe that
  $\epsilon_\ell\downarrow 0$ as $\ell\to \infty$.  By using the
  symmetry of the function $r(1-r)$ with respect to $r=1/2$ and
  performing the change of variable $r=(1+\alpha -\epsilon_\ell s)/2$
  in \eqref{s3} we deduce, recalling $J=(1-\alpha^2)/4$,
  \begin{equation*}
    \begin{split}
      &\mc G_\ell(\upbar u_\ell,\upbar \phi_\ell) -\mc G(\upbar u,\upbar\phi) 
      = 4 \, \epsilon_\ell \int_0^{\alpha/\epsilon_\ell} 
      \!\!\! \frac {1}{1-\alpha^2 + 2 \alpha \epsilon_\ell s -\epsilon_\ell^2s^2}
      \\
      &\qquad\qquad \times \Big[ \log \frac{1+2\alpha s -\epsilon_\ell
        s^2}{2\alpha s -\epsilon_\ell s^2} +\frac {1-\alpha^2
        -\epsilon_\ell}{1+2\alpha s -\epsilon_\ell s^2 } \:
      \epsilon_\ell^{-1} \log \Big( 1-
      \frac{\epsilon_\ell}{1-\alpha^2}\Big) \Big] \,ds\;.
    \end{split}
  \end{equation*}
  It is now simple to check that
  \begin{equation*}
    \lim_{\ell\to\infty} \epsilon_\ell^{-1} 
    \Big[ \mc G_\ell(\upbar u_\ell,\upbar\phi_\ell) 
    -\mc G(\upbar u,\upbar\phi) \Big]
    = \frac{4}{1-\alpha^2} \int_0^{+\infty}
    \Big[  \log \frac{1+2\alpha s}{2\alpha s}
    -\frac {1}{1+2\alpha s} \Big] \,ds 
  \end{equation*}
  and hence, computing the integral on the right hand side above, we
  get
  \begin{equation*}
    \lim_{\ell\to\infty} \epsilon_\ell^{-1} 
    \Big[ \mc G_\ell(\upbar u_\ell,\upbar\phi_\ell) 
    -\mc G(\upbar u,\upbar\phi) \Big]
    = \frac {2}{\alpha(1-\alpha^2)} \: ,
  \end{equation*}
  which, in view of Lemma~\ref{t:jell} and the definition of $\epsilon_\ell$, concludes the proof.
\end{proof}

We have now collected all the ingredients needed to complete the proof
of the development by $\Gamma$-convergence.

\begin{proof}[Proof of Theorem \ref{sviluppo}]
  Given $z\in \bb R$ and $\ell>|z|$, we introduce the function
   $u^{(z)}_\ell :\bb R\to (0,1)$ defined by
  \begin{equation}
    \label{recseq}
    u^{(z)}_\ell(x) :=
    \begin{cases}
      u_-   & \textrm{ for $x\in (-\infty,-\ell)$},\\
      \upbar u_{\ell+z} (x-z) & \textrm{ for $x\in [-\ell,z)$},\\
      \upbar u_{\ell-z} (x-z) & \textrm{ for $x\in [z,\ell)$},\\
      u_+ & \textrm{ for $x\in [\ell,+\infty)$.}
    \end{cases}
  \end{equation}
  Moreover, we set $\phi^{(z)}_\ell := s'(u^{(z)}_\ell)$.  
  Observe that $u^{(z)}_\ell$ is a continuous piecewise smooth
  function and $u^{(0)}_\ell=\upbar u_\ell$.  Moreover, by
  construction we clearly have
  \begin{equation}
    \label{conslemma}
    \mc G_\ell (u^{(z)}_\ell, \phi^{(z)}_\ell) 
    = \frac 12 \Big[ 
    \mc G_{\ell+z} (\upbar u_{\ell + z},\upbar\phi_{\ell + z}) 
    + \mc G_{\ell-z} (\upbar u_{\ell - z},\upbar\phi_{\ell - z}) \Big]\;.
  \end{equation}
  Finally, as it is easy to verify 
  \begin{equation}
    \label{minuz}
    \mc G_\ell (u^{(z)}_\ell, \phi^{(z)}_\ell) 
    = \min \big\{ \mc G_\ell (u,\phi) \,, \: 
    (u,\phi) \in \mc X\,:\: \phi (z) = 0 \big\}\;.  
  \end{equation}

  \smallskip
  \noindent \emph{Proof of \emph{(i)}.}  Fix a sequence
  $\{(u_\ell,\phi_\ell)\}\subset \mc X$ such that
  $\limsup_{\ell\to\infty} \hat{\mc G_\ell}(u_\ell,\phi_\ell)<+\infty$
  and let $z_\ell\in (-\ell,\ell)$ be such that $\phi_\ell(z_\ell) =
  0$.  From \eqref{minuz} we deduce
  \begin{equation*}
    \limsup_{\ell\to\infty} e^{\alpha \ell} 
    \big[ \mathcal {G}_\ell (u_\ell^{(z_\ell)},\phi_\ell^{(z_\ell)}) 
    - \mc G(\upbar u, \upbar\phi) \big] 
    \le  \limsup_{\ell\to\infty} \hat{\mc G_\ell}(u_\ell,\phi_\ell)
    < +\infty\;.
  \end{equation*}
  By \eqref{conslemma} we thus have
  \begin{equation*}
    \limsup_{\ell\to\infty} e^{\alpha \ell} 
    \big[ \mathcal{G}_{\ell+z_\ell} 
    (\upbar u_{\ell+z_\ell},\upbar \phi_{\ell + z_\ell}) 
    -  \mc G(\upbar u, \upbar \f) 
    +  \mathcal {G}_{\ell-z_\ell}
    (\upbar u_{\ell - z_\ell},\upbar \phi_{\ell -z_\ell}) 
    -  \mc G(\upbar u, \upbar \f) \big] < +\infty. 
  \end{equation*}
  This bound together with Lemma~\ref{lemma} yields the
  uniform boundedness of the sequence $\{z_\ell\}$. Hence,
  by taking a subsequence, $z_\ell \to z$ for some $z\in \bb R$.
  From the equi-coercivity modulo translations in item (ii) of 
  Theorem~\ref{t:2} we then deduce the precompactness of
  the sequence  $\{(u_\ell,\f_\ell)\}$. Since its limit points are
  necessarily minimizers of $\mc G$, we get 
  $(u_\ell,\f_\ell)\to (\tau_{z} \upbar u,\tau_{z} \upbar
  \phi)$.

  \smallskip
  \noindent \emph{Proof of \emph{(ii)}.} 
  In order to prove the $\Gamma$-limsup inequality, let $(u,\f)\in \bb
  R$, and assume without loss of generality that $(u,\phi)=(\tau_z
  \upbar u, \tau_z\upbar\phi)$ for some $z\in\bb R$. We claim that
  $\{(u^{(z)}_\ell, \phi^{(z)}_\ell)\}$, with $u^{(z)}_\ell $ as
  defined in \eqref{recseq} and $\phi^{(z)}_\ell= s'(u^{(z)}_\ell)$,
  is a recovery sequence.  Indeed, from \eqref{conslemma} and Lemma
  \ref{lemma} it immediately follows that
  \begin{equation}
    \label{gls}
      \begin{split}
     & \lim_{\ell\to\infty} 
     \hat{\mc G}_\ell(u_\ell^{(z)},\phi_\ell^{(z)})
     = \lim_{\ell\to\infty} e^{\alpha\ell}
    \Big[ \mc G_\ell(u_\ell^{(z)},\phi_\ell^{(z)}) -\mc G(\upbar
    u,\upbar \f) \Big]\\
    &\quad = \frac 12 \lim_{\ell\to\infty} e^{\alpha \ell} 
    \big[ \mathcal{G}_{\ell+z} (\upbar u_{\ell + z},\upbar \phi_{\ell+ z}) 
    - \mc G(\upbar u, \upbar \f) +
    \mathcal {G}_{\ell-z} (\upbar u_{\ell - z},\upbar \phi_{\ell-z}) 
    -  \mc G(\upbar u, \upbar \f) \big]
    \\
    &\quad = \frac {8\alpha}{1-\alpha^2} \frac{e^{\alpha z}
      +e^{-\alpha z}}{2}
    = \frac {8\alpha}{1-\alpha^2} \cosh(\alpha z)\;.
  \end{split}
  \end{equation}

  To prove the $\Gamma$-liminf inequality, let $(u_\ell,\f_\ell) \to
  (u,\f)\in \mc X$. By the compactness properties stated in item (ii)
  of Theorem \ref{t:2}, we can assume without loss of generality that 
  $(u,\phi) = (\tau_z \upbar u,\tau_z \upbar \f)$ for some $z\in\bb
  R$. 
  Let $z_\ell \in \f_\ell^{-1}(\{0\})$, $\ell>0$.  Since $\phi_\ell$
  converges uniformly to $\f$, we have $z_\ell\to z$.
  Observe that \eqref{gls} holds also if $z$ on the left hand side 
  is replaced by a sequence $z_\ell\to z$. Therefore, by \eqref{minuz} 
  we deduce
  \begin{equation*}
    \begin{split} 
      \liminf_{\ell\to\infty} 
     \hat{\mc G}_\ell(u_\ell,\phi_\ell)
     &= \liminf_{\ell\to\infty} e^{\alpha \ell} 
     \big[ \mathcal{G}_\ell (u_{\ell},\phi_{\ell} ) - \mc G( u, \f)
     \big] 
     \\
     &
     \ge \liminf_{\ell\to\infty} 
     e^{\alpha \ell} 
     \big[ \mathcal {G}_\ell (u^{ (z_\ell)}_{\ell},
     \phi^{(z_\ell)}_{\ell}) -  \mc G( u,  \f) \big] 
    = \frac {8\alpha}{1-\alpha^2} \cosh(\alpha z),
  \end{split}
  \end{equation*}
  that proves the $\Gamma$-liminf inequality.
\end{proof}

\end{document}